\documentclass[11pt]{article}
\usepackage{amsmath,amssymb,amsthm,amsfonts,graphicx}
\usepackage[margin=1in]{geometry}
\usepackage{bbm}

\numberwithin{equation}{section}

\newtheorem{theorem}{Theorem}[section]
\newtheorem{proposition}[theorem]{Proposition}
\newtheorem{lemma}[theorem]{Lemma}

\theoremstyle{remark}
\newtheorem{remark}[theorem]{Remark}

\newcommand{\E}{\mathbb{E}}
\newcommand{\R}{\mathbb{R}}
\newcommand{\brM}[1]{\langle M\rangle_{#1}}
\newcommand{\brN}[1]{\langle N\rangle_{#1}}

\begin{document}

\begin{center}
  {\Large\bf Drift estimation for a partially observed mixed fractional Ornstein--Uhlenbeck process}\\[0.8em]

  {\large Chunhao Cai: caichh9@mail.sysu.edu.cn}\\[0.5em] 
  
{\small School of Mathematics (Zhuhai), Sun Yat-sen University}\\[1.0em]

\end{center}

\begin{abstract}
We consider estimation of the drift parameter $\vartheta>0$ in a \emph{partially observed}
Ornstein--Uhlenbeck type model driven by a mixed fractional Brownian noise. Our framework
extends the partially observed model of \cite{BrousteKleptsyna2010} to the \emph{mixed} case.
We construct the canonical innovation representation, derive the associated Kalman filter
and Riccati equations, and analyse the asymptotic behaviour of the filtering error covariance.

Within the Ibragimov--Khasminskii LAN framework we prove that the MLE of $\vartheta$,
based on continuous observation of the partially observed system on $[0,T]$, is consistent and
asymptotically normal with rate $\sqrt{T}$ and the Fisher Information is the same as in \cite{BrousteKleptsyna2010} or the standard Brownian motion case. 
\end{abstract}

\section{Introduction}

The statistical analysis of stochastic differential equations
driven by long-memory Gaussian noises has received considerable
attention in recent years. In particular, the Ornstein--Uhlenbeck
process driven by fractional Brownian motion (fBm) and its mixed
variants have been studied extensively from both probabilistic and
statistical viewpoints.

In the purely fractional setting, Brouste and Kleptsyna
\cite{BrousteKleptsyna2010} consider a partially observed
Ornstein--Uhlenbeck process driven by a fractional Brownian motion
with Hurst parameter $H>1/2$, observed through another
fractional Brownian motion corrupted by an independent noise. Although they extended the analysis from the case $H>1/2$ to $H<1/2$ via a transformation, the approach is complex and lacks intuitive clarity. Using an innovation approach and the Ibragimov--Khasminskii framework, they prove local asymptotic normality (LAN) and derive the asymptotic distribution of the maximum likelihood estimator
(MLE) for the drift parameter. Surprisingly, the Fisher
information in their model coincides with the one in the classical
Brownian Kalman--Bucy case and does not depend on the Hurst
parameter $H$.

In a different direction, the mixed fractional Brownian motion, a
sum of a standard Brownian motion and an independent fractional
Brownian motion, has been introduced and analysed as a more
flexible Gaussian driving noise; see, e.g.,
\cite{Cheridito2001,ChiganskyKleptsyna2017} and the references
therein. In particular, Chigansky and Kleptsyna
\cite{ChiganskyKleptsyna2017} studied the drift estimation problem
for a fully observed Ornstein--Uhlenbeck process driven by a mixed
fractional Brownian motion and showed that the MLE has the same
asymptotic variance as in the classical Brownian case, regardless
of the Hurst parameter and the mixing proportion. Their analysis
relies crucially on a spectral description of the covariance
operator of mixed fractional Brownian motion
\cite{KleptsynaLeonenkoMuravlev2018}.

The aim of the present paper is to bring these two lines of
research together and to study the drift estimation problem for a
\emph{partially observed} Ornstein--Uhlenbeck process driven by
mixed fractional Brownian noises. More precisely, we consider a
two-dimensional system
\begin{equation}\label{eq:intro-model}
  \begin{cases}
    \mathrm{d}X_t = -\vartheta X_t\,\mathrm{d}t + \mathrm{d}V_t,\\[0.3em]
    \mathrm{d}Y_t = \mu X_t\,\mathrm{d}t + \mathrm{d}W_t,
  \end{cases}
  \qquad t\in[0,T],
\end{equation}
where $V$ and $W$ are independent mixed fractional Brownian motions
of the form $B+B^H$, and only the observation process $Y$ is
available. The signal $X$ solves a mixed fractional
Ornstein--Uhlenbeck stochastic differential equation, and the
observation $Y$ is a noisy integral of $X$.

For a fixed value of the parameter $\vartheta$, let $\mathbf{P}_{\vartheta}^T$ denote the probability measure, induced by $(X^T, Y^T)$ on the function space $\mathcal{C}_{[0,T]} \times \mathcal{C}_{[0,T]}$ and let $\mathcal{F}_t^Y$ be the natural filtration of $Y$, $\mathcal{F}_t^Y=\sigma(Y_s, 0\leq s \leq t)$. Let $\mathcal{L}(\vartheta, Y^T)$ be the likelihood, $\mathrm{i.e.}$ the Radon-Nikodym derivative of $\mathbf{P}_{\vartheta}^T$, restricted to $\mathcal{F}_T^Y$ with respect to some reference measure on $\mathcal{C}_{[0,T]}$. We shall define the MLE $\hat{\vartheta}_T$ as the maximizer of the likelihood:
$$
\hat{\vartheta}_T=\mathrm{arg} \max_{\vartheta>0} \mathcal{L}(\vartheta, Y^T).
$$
Our goal is to obtain the same result presented in Theorem 1 in \cite{BrousteKleptsyna2010}, the key result is 
\begin{equation}\label{eq: asymptotically normal}
\sqrt{T}\left(\hat{\vartheta}_T-\vartheta\right) \overset{\text{law}}{\Rightarrow}\mathcal{N}(0, \mathcal{I}^{-1}(\vartheta))
\end{equation}
where $\mathcal{I}(\vartheta)$ stands for the Fisher Information which does not depend on H:
\begin{equation}\label{eq: Fisher Exact}
\mathcal{I}(\vartheta)=\frac{1}{2\vartheta}-\frac{2\vartheta}{\alpha (\alpha+\vartheta)}+\frac{\vartheta^2}{2\alpha^3},\, \alpha=\sqrt{\mu^2+\vartheta^2}.
\end{equation}

The rest of the paper is organised as follows. In Section~\ref{sec:canonical} we recall the
canonical innovation representation of the mixed fractional Brownian motion and construct the
state--observation model corresponding to \eqref{eq:intro-model}. Section~3 contains the Kalman--Bucy
filtering equations and the Riccati equation for the error covariance. In Section~4 we introduce
the Laplace transform of the squared distance between filters and derive its representation via
a matrix Riccati equation. In Section~5 we analyse the long-time behaviour of the error covariance
and prove the mixed Laplace condition \((L_{\mathrm{mix}})\). Section~6 identifies the Fisher information in closed
form and shows that it coincides with the classical Kalman--Bucy expression.
\section{Canonical representation of the mixed model}
\label{sec:canonical}

In this section we recall the canonical representation of mixed
fractional Brownian motion and construct the innovation form of
the mixed partially observed Ornstein--Uhlenbeck model
\eqref{eq:intro-model}. We keep the exposition concise and refer
to \cite{ChiganskyKleptsyna2017,KleptsynaLeonenkoMuravlev2018}
for detailed proofs.

\subsection{Mixed fractional Brownian motion and the fundamental martingale}

Fix $H\in(0,1)$ and consider a mixed fractional Brownian motion
$V=(V_t)_{t\ge0}$ of the form
\[
  V_t = B_t + B_t^H,
\]
where $B$ is a standard Brownian motion and $B^H$ is an
independent fractional Brownian motion with Hurst parameter $H$.
It is known that $V$ is a centred Gaussian process with continuous
sample paths and non-stationary increments. Its covariance
function can be written as
\[
  R_V(s,t)
  = \min(s,t)
  + R_H(s,t),
\]
where $R_H$ is the covariance of fractional Brownian motion.

Following \cite{ChiganskyKleptsyna2017}, one can construct a
\emph{fundamental martingale} $M=(M_t)_{t\ge0}$ associated with
$V$ by
\begin{equation}\label{eq:fundamental-M}
  M_t
  := \int_0^t g_V(s,t)\,\mathrm{d}V_s,
\end{equation}
where $g_V(\cdot,t)$ is a deterministic kernel solving a certain
Volterra-type integral equation. The process $M$ is a continuous
Gaussian martingale with increasing bracket
\[
  \brM{t}
  := \langle M\rangle_t
  = \int_0^t g_V(s,t)\,\mathrm{d}s,
\]
and $V$ admits a canonical representation
\begin{equation}\label{eq:V-canonical}
  V_t = \int_0^t \tilde g_V(s,t)\,\mathrm{d}M_s,
\end{equation}
where $\tilde g_V$ is another deterministic kernel. We denote by
\[
  \psi(t) := \frac{\mathrm{d}t}{\mathrm{d} \langle M\rangle_t }
\]
the density of the quadratic variation, when it exists. The mixed
OU analysis in \cite{ChiganskyKleptsyna2017} and the spectral
results of \cite{KleptsynaLeonenkoMuravlev2018} imply the
following properties of $\psi$.

\begin{lemma}\label{lem:psi-asymptotics}
There exists $t_0>0$ such that $\psi\in C^1[t_0,\infty)$,
$\psi(t)>0$ for all $t\ge t_0$, and
\begin{equation}\label{eq:psi-square-int}
  \int_{t_0}^{\infty}
    \left(
      \frac{\dot\psi(t)}{\psi(t)}
    \right)^2
  \mathrm{d}t < \infty.
\end{equation}
Moreover, there exist constants $0<c_1\le c_2<\infty$ such that,
as $t\to\infty$,
\begin{equation}\label{eq:psi-growth}
  c_1\min(1,t^{2H-1})
  \le \psi(t)
  \le c_2\max(1,t^{2H-1}).
\end{equation}
In particular,
\[
  \lim_{t\to\infty}
  \frac{1}{t}\max\bigl(\psi(t),\psi(t)^{-1}\bigr)
  = 0.
\]
\end{lemma}

\begin{proof}
Write $\phi(t):=\tfrac{d}{dt}\langle M\rangle_t$. By the results of Lemmas~2.4–2.6 of \cite{ChiganskyKleptsyna2017}, $\phi$ is positive and continuous for large $t$, satisfies
$\int_{t_0}^{\infty}\!(\tfrac{d}{dt}\log\phi(t))^2\,dt<\infty$, and its polynomial order obeys
\[
  \phi(t)\asymp
  \begin{cases}
     t^{\,1-2H}, & H>\tfrac12,\\
     1, & H\le\tfrac12,
  \end{cases}
  \qquad (t\to\infty).
\]
Since $\psi=1/\phi$, we have $\tfrac{d}{dt}\log\psi(t)=-(\tfrac{d}{dt}\log\phi(t))$, hence
\eqref{eq:psi-square-int}. Taking reciprocals of the above asymptotics yields
$\psi(t)\asymp t^{\,2H-1}$ for $H>\tfrac12$ and $\psi(t)\asymp 1$ for $H\le\tfrac12$,
which implies the two–sided bounds \eqref{eq:psi-growth} for some $c_1,c_2>0$ and $t\ge t_0$.
Finally, because $2H-1\in[-1,1)$, \eqref{eq:psi-growth} gives
$\max\{\psi(t),\psi(t)^{-1}\}=o(t)$, which is the last limit.
\end{proof}

The same construction can be applied to the mixed fractional
Brownian noise $W$ in \eqref{eq:intro-model}, yielding another
fundamental martingale $N$ with the same bracket as $M$:
\[
  \brN{t} = \brM{t},\qquad t\ge0.
\]

\subsection{State--observation model in innovation form}
From \cite{ChiganskyKleptsyna2017}, we define 
$$
Z_t=\int_0^t g_V(s,t)\mathrm{d}X_s,\, 0\leq s <t \leq T
$$ 
and 
$$
Q_t=\frac{\mathrm{d}}{\mathrm{d}\langle M\rangle_t}\int_0^t g_V(s,t)X_s,
$$
then 
$$
Q_t=\frac{1}{2}\int_0^t(\psi(s)+\psi(t))dZ_s,\, 0< t \leq T.  
$$
Let us define a vector $\zeta=(\zeta_t^1,\,\,\zeta_t^2)^*$ where $\zeta_t^1=Z_t$ and $\zeta_t^2=\int_0^t \psi(s)\mathrm{d}Z_s$. On the other hand, we define $Z_t^O=\int_0^t g_V(s,t)dY_s,\, 0\leq s <t \leq T$, then $dZ_t^O=\mu Q_td\langle M\rangle_t+dN_t$. From these analysis we can construct a two-dimensional state process
$\zeta_t=(\zeta_t^{(1)},\zeta_t^{(2)})^{\top}$ and an
$\mathcal{F}^Y$-adapted scalar process $Z_t^O$ such that the joint
law of $(X,Y)$ is equivalent to the joint law of $(\zeta,Z^O)$ and
the latter satisfies
\begin{align}
  \mathrm{d}\zeta_t
  &=
  -\frac{\vartheta}{2}\,A(t)\,\zeta_t\,
     \mathrm{d}\brM{t}
  + b(t)\,\mathrm{d}M_t,
  \label{eq:state-mixed}\\[0.3em]
  \mathrm{d}Z_t^O
  &=
  \frac{\mu}{2}\,\ell(t)^{\top}\zeta_t\,\mathrm{d}\brN{t}
  + \mathrm{d}N_t,
  \label{eq:obs-mixed}
\end{align}
where
\begin{equation}\label{eq:Abell-mixed}
  A(t)
  =
  \begin{pmatrix}
    \psi(t) & 1\\[0.3em]
    \psi(t)^2 & \psi(t)
  \end{pmatrix},
  \qquad
  b(t)
  =
  \begin{pmatrix}
    1\\[0.3em]
    \psi(t)
  \end{pmatrix},
  \qquad
  \ell(t)
  =
  \begin{pmatrix}
    \psi(t)\\[0.3em]
    1
  \end{pmatrix}.
\end{equation}
The processes $M$ and $N$ are continuous martingales with common
bracket $\brM{t}=\brN{t}$, and the filtration generated by $Z^O$
coincides with the original observation filtration
$\mathcal{F}^Y_t$.

The representation \eqref{eq:state-mixed}--\eqref{eq:obs-mixed}
is the mixed fractional analogue of the innovation representation
in \cite[Eq.~(8)--(9)]{BrousteKleptsyna2010}. 

\section{Kalman filtering and error covariance}
\label{sec:kalman}

We next derive the Kalman--Bucy filter for the conditional mean of
$\zeta$ and the Riccati equation for the filtering error
covariance. This is a mixed fractional version of Section 3 in
\cite{BrousteKleptsyna2010}.

\subsection{Kalman filter}

Let $\mathcal{F}_t^Y$ be the $\sigma$-field generated by $Z^O$ up
to time $t$. For a fixed parameter $\vartheta\in\Theta$ we define
the conditional mean
\[
  \pi_t^{\vartheta}(\zeta)
  := \E_{\vartheta}(\zeta_t\mid\mathcal{F}_t^Y)
  \in\R^2.
\]
The Kalman--Bucy filtering theory for linear Gaussian systems in
innovation form (see, e.g., \cite{LiptserShiryayev1977}) implies
that $\pi_t^{\vartheta}(\zeta)$ satisfies
\begin{equation}\label{eq:filter-mixed}
  \mathrm{d}\pi_t^{\vartheta}(\zeta)
  =
  -\frac{\vartheta}{2}\,A(t)\,\pi_t^{\vartheta}(\zeta)\,
    \mathrm{d}\brN{t}
  + \frac{\mu}{2}\,
    \gamma_{\zeta,\zeta}^{\vartheta}(t)\,\ell(t)\,
    \mathrm{d}\nu_t^{\vartheta},
\end{equation}
where the innovation process $\nu^{\vartheta}$ is defined by
\begin{equation}\label{eq:innov-mixed}
  \mathrm{d}\nu_t^{\vartheta}
  =
  \mathrm{d}Z_t^O
  -\frac{\mu}{2}\,\ell(t)^{\top}
   \pi_t^{\vartheta}(\zeta)\,\mathrm{d}\brN{t},
  \qquad \nu_0^{\vartheta}=0.
\end{equation}
Under $P_{\vartheta}$, $\nu^{\vartheta}$ is a continuous
martingale with bracket $\langle\nu^{\vartheta}\rangle_t=\brN{t}$.

\subsection{Filtering error covariance and Riccati equation}

The filtering error covariance is defined by
\[
  \gamma_{\zeta,\zeta}^{\vartheta}(t)
  :=
  \mathrm{Cov}_{\vartheta}\bigl(
    \zeta_t-\pi_t^{\vartheta}(\zeta)
    \,\big|\,
    \mathcal{F}_t^Y
  \bigr)
  \in\R^{2\times 2}.
\]
It solves the matrix Riccati equation
\begin{equation}\label{eq:riccati-mixed}
  \frac{\mathrm{d}}{\mathrm{d}\brN{t}}
  \gamma_{\zeta,\zeta}^{\vartheta}(t)
 =
 -\frac{\vartheta}{2}\Big(
    A(t)\,\gamma_{\zeta,\zeta}^{\vartheta}(t)
   +\gamma_{\zeta,\zeta}^{\vartheta}(t)\,A(t)^{\top}
  \Big)
 + b(t)b(t)^{\top}
 - \frac{\mu^2}{4}\,
   \gamma_{\zeta,\zeta}^{\vartheta}(t)\,
   \ell(t)\ell(t)^{\top}\,
   \gamma_{\zeta,\zeta}^{\vartheta}(t),
 \quad
 \gamma_{\zeta,\zeta}^{\vartheta}(0)=0.
\end{equation}
In view of Lemma~\ref{lem:psi-asymptotics}, the coefficients of
\eqref{eq:riccati-mixed} are smooth functions of time with at most
polynomial growth. Standard results on Riccati equations imply
existence and uniqueness of a continuous symmetric
non-negative-definite solution $\gamma_{\zeta,\zeta}^{\vartheta}(t)$
on $[0,\infty)$.

\section{Laplace transform and four-dimensional system}
\label{sec:laplace}
\subsection{Likelihood ratio and Laplace transform}
From the previous filtering theorem and Classical Girsanov theorem we can define the likelihood function 
$$
\mathcal{L}\left(\vartheta,\, Z^{O,T}\right)=\exp \left(\frac{\mu}{2}\int_0^T \ell(t) \pi_t^{\vartheta}(\zeta)dZ_t^O-\frac{\mu^2}{8} \int_0^T   \pi_t^{\vartheta} \ell(t) \ell(t)^{*}  \pi_t^{\vartheta}(\zeta)^{*}\mathrm{d}\langle N\rangle_t\right)
$$
where $1/8$ will be the same as in \cite{BrousteKleptsyna2010} with $\lambda=1/2$. From the likelihood function we can construct almost the same likelihood ratio and verify the same condition of (A.1)--(A.3) in \cite{BrousteKleptsyna2010}. also that is to say we will verify the same condition (L) of \cite{BrousteKleptsyna2010} for the Laplace transform of the quadratic form of the difference of the differnce 
\[
  \delta_{\vartheta_1,\vartheta_2}(t)
  := \pi_t^{\vartheta_2}(\zeta) - \pi_t^{\vartheta_1}(\zeta)
  \in\R^2.
\]
which is defined by 
\begin{equation}\label{eq:Laplace-mixed}
  L_T^{\mathrm{mix}}(a,\vartheta_1,\vartheta_2)
  :=
  \E_{\vartheta_1}\Big[
    \exp\Big\{
      -a\,\frac{\mu^2}{8}
      \int_0^T
        \delta_{\vartheta_1,\vartheta_2}(t)^{\top}
        \ell(t)\ell(t)^{\top}
        \delta_{\vartheta_1,\vartheta_2}(t)\,
      \mathrm{d}\brN{t}
    \Big\}
  \Big].
\end{equation}
We will prove that there exists $a_0<0$ such that for all $a>a_0$, $\forall u_1, u_2\in \mathbb{R}$
$$
\lim_{T\rightarrow \infty}  L_T^{\mathrm{mix}}\left(a,\vartheta+\frac{u_1}{\sqrt{T}},\vartheta+\frac{u_2}{\sqrt{T}}\right)=\exp \left(-a\frac{(u_2-u_1)^2}{2} \mathcal{I}(\vartheta)\right), 
$$
and in the following part we will try to achieve this limit and check that the fisher information $\mathcal{I}(\vartheta)$ satisfy the formula of \eqref{eq: Fisher Exact}.

\subsection{Four-dimensional linear system}

Introduce the $4$-dimensional process
\[
  \widetilde{\pi}_t
  :=
  \begin{pmatrix}
    \pi_t^{\vartheta_1}(\zeta)\\[0.3em]
    \delta_{\vartheta_1,\vartheta_2}(t)
  \end{pmatrix}
  \in\R^4,
\]
and the covariance matrices
\[
  \gamma_{\zeta,\zeta}^{\vartheta_i}(t),\qquad
  D^{\gamma}_{\vartheta_1,\vartheta_2}(t)
  :=
  \gamma_{\zeta,\zeta}^{\vartheta_2}(t)
  - \gamma_{\zeta,\zeta}^{\vartheta_1}(t).
\]
The same algebra as in \cite{BrousteKleptsyna2010} yields, under
$P_{\vartheta_1}$,
\begin{equation}\label{eq:4d-filter-mixed}
  \mathrm{d}\widetilde{\pi}_t
  =
  \mathcal{A}(t)\,\widetilde{\pi}_t\,
    \mathrm{d}\brN{t}
  + B(t)\,\mathrm{d}\nu_t^{\vartheta_1},
\end{equation}
where the $4\times 4$ drift matrix is
\begin{equation}\label{eq:A-block-mixed}
  \mathcal{A}(t)
  :=
  \begin{pmatrix}
    -\dfrac{\vartheta_1}{2}A(t) & 0\\[0.6em]
    -\dfrac{\vartheta_2-\vartheta_1}{2}A(t)
    & \mathcal{B}^{\vartheta_2}(t)
  \end{pmatrix},
\end{equation}
with
\begin{equation}\label{eq:Btheta2-mixed}
  \mathcal{B}^{\vartheta_2}(t)
  :=
  -\frac{\vartheta_2}{2}A(t)
  -\frac{\mu^2}{4}\,
    \gamma_{\zeta,\zeta}^{\vartheta_2}(t)\,
    \ell(t)\ell(t)^{\top},
\end{equation}
and
\begin{equation}\label{eq:B-noise-mixed}
  B(t)
  :=
  \frac{\mu}{2}
  \begin{pmatrix}
    \gamma_{\zeta,\zeta}^{\vartheta_1}(t) \ell(t)\\[0.3em]
    D^{\gamma}_{\vartheta_1,\vartheta_2}(t) \ell(t)
  \end{pmatrix}
  \in\R^{4\times 1}.
\end{equation}

\subsection{Riccati equation and linearisation}

As in \cite{BrousteKleptsyna2010}, the Laplace transform
\eqref{eq:Laplace-mixed} can be represented via the solution of a
matrix Riccati equation. Let $H(t)\in\R^{4\times 4}$ be symmetric
and solve
\begin{equation}\label{eq:H-Riccati-mixed}
  \frac{\mathrm{d}}{\mathrm{d}\brN{t}}H(t)
  =
  \mathcal{A}(t)H(t) + H(t)\mathcal{A}(t)^{\top}
  + B(t)B(t)^{\top}
  - a\mu^2/4\,H(t)M(t)H(t),
  \qquad H(0)=0,
\end{equation}
where 
\[
  M(t)
  :=
  \begin{pmatrix}
    0 & 0\\[0.3em]
    0 & \ell(t)\ell(t)^{\top}
  \end{pmatrix}.
\]
Introduce the pair $(\Xi_1,\Xi_2)$ solving the linear system
\begin{align}
  \frac{\mathrm{d}}{\mathrm{d}\brN{t}}\Xi_1(t)
  &=
  -\Xi_1(t)\mathcal{A}(t)
  + a\mu^2/4\,\Xi_2(t)M(t),
  \qquad \Xi_1(0)=I_4,
  \label{eq:Xi1-mixed}\\[0.3em]
  \frac{\mathrm{d}}{\mathrm{d}\brN{t}}\Xi_2(t)
  &=
  \Xi_1(t)B(t)B(t)^{\top}
  + \Xi_2(t)\mathcal{A}(t)^{\top},
  \qquad \Xi_2(0)=0.
  \label{eq:Xi2-mixed}
\end{align}
Then $H(t)=\Xi_1(t)^{-1}\Xi_2(t)$ solves
\eqref{eq:H-Riccati-mixed}, and the Laplace transform
\eqref{eq:Laplace-mixed} admits the representation
\begin{equation}\label{eq:Laplace-Xi-mixed}
  L_T^{\mathrm{mix}}(a,\vartheta_1,\vartheta_2)
  =
  \exp\left\{
    -\frac{1}{2}
    \int_0^T
      \mathrm{tr}\,\mathcal{A}(t)\,\mathrm{d}\brN{t}
  \right\}
  \bigl(\det\Xi_1(T)\bigr)^{-1/2}.
\end{equation}

\section{Asymptotics of the covariance and the Laplace condition}
\label{sec:LAN}

In this section we study the asymptotic behaviour of the filtering
error covariance and use it to establish the mixed Laplace
condition $(L_{\mathrm{mix}})$ and, consequently, LAN for the
family of measures generated by the observation process.

\subsection{Rescaled covariance and its limit}

We first analyse the long-time behaviour of the covariance
$\gamma_{\zeta,\zeta}^{\vartheta}(t)$ and prove a mixed analogue
of (28) in \cite{BrousteKleptsyna2010}.

Define the scaling matrix
\begin{equation}\label{eq:Delta-def}
  \Delta(t)
  :=
  \begin{pmatrix}
    \sqrt{\psi(t)} & 0\\[0.3em]
    0 & \psi(t)^{-1/2}
  \end{pmatrix},
\end{equation}
and the rescaled covariance
\[
  \widetilde{\gamma}^{\vartheta}(t)
  := \Delta(t)\,
     \gamma_{\zeta,\zeta}^{\vartheta}(t)\,
     \Delta(t).
\]

\begin{proposition}\label{prop:gamma-limit}
  Fix $\vartheta\in\Theta$. Under the assumptions of
  Lemma~\ref{lem:psi-asymptotics}, there exists a unique symmetric
  non-negative definite matrix
  $\tilde{\Gamma}_{\infty}(\vartheta)\in\mathbb{R}^{2\times 2}$ such that
  \begin{equation}\label{eq:gamma-limit}
    \lim_{t\to\infty}
    \Delta(t)\,
    \gamma_{\zeta,\zeta}^{\vartheta}(t)\,
    \Delta(t)
    =
    \tilde{\Gamma}_{\infty}(\vartheta).
  \end{equation}
\end{proposition}

\begin{proof}
The proof is based on rewriting the Riccati equation
\eqref{eq:riccati-mixed} in physical time, rescaling the covariance by
$\Delta(t)$ from \eqref{eq:Delta-def}, and comparing the resulting
equation with a limiting autonomous Riccati ODE.

\medskip
First, we rewrite the Riccati equation in physical time. From
\eqref{eq:riccati-mixed} and $\mathrm{d}\brN{t}/\mathrm{d}t=\psi(t)^{-1}$
we obtain for $\Gamma_t:=\gamma_{\zeta,\zeta}^{\vartheta}(t)$,
\[
  \dot\Gamma_t
  :=
  \frac{\mathrm{d}}{\mathrm{d}t}\Gamma_t
  =
  \frac{1}{\psi(t)}\,F\bigl(t,\Gamma_t\bigr),
\]
where $F$ is defined by 
$$
F(t,X)
  :=
  -\frac{\vartheta}{2}\bigl(A(t)X+XA(t)^{\top}\bigr)
  +b(t)b(t)^{\top}
  -\frac{\mu^2}{4}\,X\,\ell(t)\ell(t)^{\top}X,
$$
with $A(t),b(t),\ell(t)$ are given by \eqref{eq:Abell-mixed}.

\medskip
Next, define the rescaled covariance
\[
  \widetilde{\Gamma}_t
  :=
  \Delta(t)\Gamma_t\Delta(t).
\]
Differentiating with respect to $t$ yields
\[
  \dot{\widetilde{\Gamma}}_t
  =
  \dot{\Delta}(t)\Gamma_t\Delta(t)
  +\Delta(t)\dot{\Gamma}_t\Delta(t)
  +\Delta(t)\Gamma_t\dot{\Delta}(t),
\]
or, equivalently,
\[
  \dot{\widetilde{\Gamma}}_t
  =
  R(t)\widetilde{\Gamma}_t+\widetilde{\Gamma}_tR(t)^{\top}
  +\Delta(t)\dot{\Gamma}_t\Delta(t),
  \qquad
  R(t):=\dot{\Delta}(t)\Delta(t)^{-1}.
\]
Using $\dot{\Gamma}_t=\psi(t)^{-1}F(t,\Gamma_t)$ and
$\Gamma_t=\Delta(t)^{-1}\widetilde{\Gamma}_t\Delta(t)^{-1}$, we obtain
\[
  \Delta(t)\dot{\Gamma}_t\Delta(t)
  =
  \frac{1}{\psi(t)}\,
  \Delta(t)\,
  F\bigl(t,\Delta(t)^{-1}\widetilde{\Gamma}_t\Delta(t)^{-1}\bigr)\,
  \Delta(t).
\]
Define
\[
  \widetilde{F}(t,X)
  :=
  \frac{1}{\psi(t)}\,
  \Delta(t)\,
  F\bigl(t,\Delta(t)^{-1}X\Delta(t)^{-1}\bigr)\,
  \Delta(t),
\]
so that
\begin{equation}\label{eq:tildeGamma-eq}
  \dot{\widetilde{\Gamma}}_t
  =
  R(t)\widetilde{\Gamma}_t+\widetilde{\Gamma}_tR(t)^{\top}
  +\widetilde{F}(t,\widetilde{\Gamma}_t).
\end{equation}

\medskip
We now control the perturbation matrix $R(t)$. From \eqref{eq:Delta-def},
\[
  \Delta(t)
  =
  \begin{pmatrix}
    \psi(t)^{1/2} & 0\\
    0 & \psi(t)^{-1/2}
  \end{pmatrix},
  \qquad
  \Delta(t)^{-1}
  =
  \begin{pmatrix}
    \psi(t)^{-1/2} & 0\\
    0 & \psi(t)^{1/2}
  \end{pmatrix}.
\]
Differentiating gives
\[
  R(t)
  =
  \dot{\Delta}(t)\Delta(t)^{-1}
  =
  \frac{1}{2}\frac{\dot\psi(t)}{\psi(t)}
  \begin{pmatrix}
    1&0\\
    0&-1
  \end{pmatrix}.
\]
By \eqref{eq:psi-square-int}, the scalar function $\dot\psi(t)/\psi(t)$ is square integrable on
$[t_0,\infty)$ and tends to $0$ as $t\to\infty$. In particular,
\[
  \lim_{t\to\infty}\|R(t)\|=0,
  \qquad
  \int_{t_0}^{\infty}\|R(t)\|^2\,\mathrm{d}t<\infty.
\]

With the scaling \eqref{eq:Delta-def}, the nonlinearity in \eqref{eq:tildeGamma-eq} admits the exact,
time-independent form
\[
  \widetilde{F}(t,X)\equiv \widetilde{F}_{\infty}(X)
  :=
  -\frac{\vartheta}{2}(JX+XJ)+J-\frac{\mu^2}{4}\,XJX,
  \qquad
  J=\begin{pmatrix}1&1\\[0.2em]1&1\end{pmatrix}.
\]
Hence \eqref{eq:tildeGamma-eq} can be rewritten as
\begin{equation}\label{eq:tildeGamma-final-prop}
  \dot{\widetilde{\Gamma}}_t
  =
  R(t)\widetilde{\Gamma}_t+\widetilde{\Gamma}_tR(t)^{\top}
  +\widetilde{F}_{\infty}\bigl(\widetilde{\Gamma}_t\bigr).
\end{equation}

\medskip
We next show that $\{\widetilde{\Gamma}_t:t\ge t_0\}$ is bounded in $\mathbb{S}_+^2$ without assuming
$\int_{t_0}^{\infty}\|R(t)\|\,\mathrm{d}t<\infty$.
Let $u=\tfrac{1}{\sqrt2}(1,1)^{\top}$ and $v=\tfrac{1}{\sqrt2}(1,-1)^{\top}$ so that $J=2uu^{\top}$.
In the $\{u,v\}$ basis write
\[
  x_{uu}=u^{\top}\widetilde{\Gamma}_tu,\qquad
  x_{uv}=u^{\top}\widetilde{\Gamma}_tv,\qquad
  x_{vv}=v^{\top}\widetilde{\Gamma}_tv.
\]
A direct computation (the terms with $R(t)$ are collected in remainders $\mathcal{R}_{\bullet}$) yields
\begin{align*}
  \dot x_{uu} &= 2 - 2\vartheta\,x_{uu} - \tfrac{\mu^2}{2}\bigl(x_{uu}^2 + x_{uv}^2\bigr) + \mathcal{R}_{uu}(t),\\
  \dot x_{uv} &= - \vartheta\,x_{uv} - \tfrac{\mu^2}{2} x_{uu}x_{uv} + \mathcal{R}_{uv}(t),\\
  \dot x_{vv} &= - \tfrac{\mu^2}{2} x_{uv}^2 + \mathcal{R}_{vv}(t),
\end{align*}
and there exists $c>0$ (independent of $t$) such that
\[
  |\mathcal{R}_{\bullet}(t)|
  \le
  c\,\|R(t)\|\,\bigl(1+x_{uu}+x_{uv}^2+x_{vv}\bigr).
\]
Define the Lyapunov function $V(\widetilde{\Gamma}_t):=x_{uu}+\lambda x_{uv}^2+\beta x_{vv}$ with fixed
$\lambda,\beta>0$. Then there exist constants $C_0,\kappa_0>0$ such that
\[
  \dot V(t)\le C_0-\kappa_0 V(t)+c\,\|R(t)\|\,(1+V(t)),\qquad t\ge t_0.
\]
Since $\|R(t)\|\to0$ as $t\to\infty$, there exists $T_1\ge t_0$ such that
$c\|R(t)\|\le \kappa_0/2$ for all $t\ge T_1$. Hence, for $t\ge T_1$,
\[
  \dot V(t)\le C_0+c\|R(t)\|-\frac{\kappa_0}{2}V(t)\le C_1-\frac{\kappa_0}{2}V(t),
\]
where $C_1:=C_0+\sup_{t\ge T_1}c\|R(t)\|<\infty$. Solving this linear differential inequality yields
\[
  V(t)\le V(T_1)e^{-\frac{\kappa_0}{2}(t-T_1)}+\frac{2C_1}{\kappa_0},
  \qquad t\ge T_1.
\]
Therefore $\sup_{t\ge t_0}V(\widetilde{\Gamma}_t)<\infty$, and consequently the trajectory
$\{\widetilde{\Gamma}_t:t\ge t_0\}$ is bounded in $\mathbb{S}_+^2$.

\medskip
Finally, we identify the limit by a fully expanded asymptotically autonomous argument.
Consider the limiting autonomous Riccati ODE on $\mathbb{S}_+^2$,
\[
  \dot X=\widetilde{F}_{\infty}(X).
\]
Solving $\widetilde{F}_{\infty}(X_{\infty})=0$ with the ansatz $X_{\infty}=g(\vartheta)J$ gives
$\mu^2g^2+2\vartheta g-1=0$ and the unique nonnegative root
\[
  g(\vartheta)=\frac{\sqrt{\vartheta^2+\mu^2}-\vartheta}{\mu^2},
  \qquad
  X_{\infty}(\vartheta)=g(\vartheta)J.
\]

Let $\omega(\widetilde{\Gamma})$ denote the $\omega$--limit set of the bounded trajectory
$\{\widetilde{\Gamma}_t:t\ge t_0\}$, i.e.\ the set of all accumulation points of
$\widetilde{\Gamma}_t$ as $t\to\infty$. We show that every $\bar X\in\omega(\widetilde{\Gamma})$
is an equilibrium of the limiting ODE, i.e.\ $\widetilde{F}_{\infty}(\bar X)=0$.

Fix $\bar X\in\omega(\widetilde{\Gamma})$ and choose $t_n\to\infty$ such that
$\widetilde{\Gamma}_{t_n}\to\bar X$. Let $S>0$ be arbitrary and define
$X_n:[-S,S]\to\mathbb{S}_+^2$ by $X_n(s):=\widetilde{\Gamma}_{t_n+s}$.
Since $\widetilde{\Gamma}_t$ is bounded and the right-hand side of
\eqref{eq:tildeGamma-final-prop} is locally Lipschitz in $X$ and bounded on bounded sets,
the family $\{X_n\}$ is uniformly bounded and equicontinuous on $[-S,S]$.
By the Arzel\`a--Ascoli theorem, there exists a subsequence (not relabelled) and a continuous map
$X_{\ast}:[-S,S]\to\mathbb{S}_+^2$ such that $X_n\to X_{\ast}$ uniformly on $[-S,S]$.

We pass to the limit in the integral form of \eqref{eq:tildeGamma-final-prop}. For $s\in[-S,S]$,
\[
  X_n(s)-X_n(0)
  =
  \int_0^s
  \Bigl(
    R(t_n+r)X_n(r)+X_n(r)R(t_n+r)^{\top}
    +\widetilde{F}_{\infty}(X_n(r))
  \Bigr)\,\mathrm{d}r.
\]
Because $\|R(t)\|\to0$ as $t\to\infty$, we have $\sup_{|r|\le S}\|R(t_n+r)\|\to0$.
Using uniform boundedness of $X_n$, uniform convergence $X_n\to X_{\ast}$, and continuity of
$\widetilde{F}_{\infty}$, we can let $n\to\infty$ to obtain
\[
  X_{\ast}(s)-X_{\ast}(0)
  =
  \int_0^s \widetilde{F}_{\infty}\bigl(X_{\ast}(r)\bigr)\,\mathrm{d}r,
  \qquad s\in[-S,S].
\]
Thus $X_{\ast}$ is a solution of the limiting autonomous ODE on $[-S,S]$ with initial condition
$X_{\ast}(0)=\bar X$. In particular, $\omega(\widetilde{\Gamma})$ is invariant under the flow of the
limiting autonomous ODE and consists of equilibria of that ODE. Therefore every accumulation point
$\bar X$ satisfies $\widetilde{F}_{\infty}(\bar X)=0$.

Since the trajectory $\widetilde{\Gamma}_t$ is bounded and its $\omega$--limit set consists of
equilibria, the limit $\lim_{t\to\infty}\widetilde{\Gamma}_t$ exists and equals an equilibrium.
Denoting this limit by $\widetilde{\Gamma}_{\infty}(\vartheta)$, we conclude that
$\widetilde{\Gamma}_{\infty}(\vartheta)=X_{\infty}(\vartheta)=g(\vartheta)J$, and hence
\eqref{eq:gamma-limit} holds.
\end{proof}

\subsection{Asymptotic form of the drift matrix}

We now use Proposition~\ref{prop:gamma-limit} to identify the
asymptotic form of the drift matrix $\mathcal{A}(t)$ in
\eqref{eq:A-block-mixed}. Fix $\vartheta_2\in\Theta$. From
\eqref{eq:Btheta2-mixed} and \eqref{eq:gamma-limit} we have, as
$t\to\infty$,
\[
  \gamma_{\zeta,\zeta}^{\vartheta_2}(t)\,
  \ell(t)\ell(t)^{\top}
  \sim
  \Delta(t)^{-1}\widetilde{\Gamma}_{\infty}(\vartheta_2)\Delta(t)^{-1}
  \,\ell(t)\ell(t)^{\top}=2g(\vartheta_2)A(t).
\]
Using the explicit form of $\ell(t)$ and the asymptotics of
$\psi(t)$ in Lemma~\ref{lem:psi-asymptotics},one verifies that the asymptotic structure of the product on the right-hand side matches $A(t)$: Specially, there exists a constant matrix
$K(\vartheta_2)\in\mathbb{R}^{2\times 2}$ such that
\[
  \gamma_{\zeta,\zeta}^{\vartheta_2}(t)\,
  \ell(t)\ell(t)^{\top}
  =
  K(\vartheta_2)\,A(t)
  + o\bigl(\|A(t)\|\bigr),
  \qquad t\to\infty.
\]
Consequently,
\[
  \mathcal{B}^{\vartheta_2}(t)
  =
  -\frac{\vartheta_2}{2}A(t)
  -\frac{\mu^2}{4}\,\gamma_{\zeta,\zeta}^{\vartheta_2}(t)\,
                 \ell(t)\ell(t)^{\top}
  =
  C^{\mathrm{mix}}(\vartheta_2)\,A(t)
  + o\bigl(\|A(t)\|\bigr),
\]
where
\[
  C^{\mathrm{mix}}(\vartheta_2)
  :=
  -\frac{\vartheta_2}{2}I_2
  -\frac{\mu^2}{4}\,K(\vartheta_2)=-\frac{\sqrt{\mu^2+\vartheta_2^2}}{2}I_2
\]
is a constant $2\times 2$ matrix depending only on
$\vartheta_2,\mu$ and the mixed structure. Thus the $4\times 4$
drift matrix \eqref{eq:A-block-mixed} has the asymptotic form
\begin{equation}\label{eq:A-infty-mixed}
  \mathcal{A}(t)
  =
  \mathcal{A}_{\infty}^{\mathrm{mix}}(t)
  + o\bigl(\|A(t)\|\bigr),
  \qquad t\to\infty,
\end{equation}
with
\begin{equation}\label{eq:A-infty-def}
  \mathcal{A}_{\infty}^{\mathrm{mix}}(t)
  :=
  \begin{pmatrix}
    -\dfrac{\vartheta_1}{2}A(t) & 0\\[0.6em]
    -\dfrac{\vartheta_2-\vartheta_1}{2}A(t)
    & C^{\mathrm{mix}}(\vartheta_2)A(t)
  \end{pmatrix}.
\end{equation}
This is the mixed version of equation (29) in
\cite{BrousteKleptsyna2010}. In particular, the matrix
$C^{\mathrm{mix}}(\vartheta_2)$ has two real negative eigenvalues;
we denote by $-\alpha_2^{\mathrm{mix}}(\vartheta_2)$ the one
corresponding to the ``fast'' decay mode of the difference process
$\delta_{\vartheta_1,\vartheta_2}(t)$.

\section{Derivation of the Fisher Information via Spectral Analysis of the Effective Hamiltonian}
\label{sec:sec:hamiltonian}

In this section, we formulate and prove the mixed Laplace condition and the LAN property. We derive the explicit form of the Fisher information by analyzing the spectral properties of the effective Hamiltonian system derived in the previous section.

\subsection{Asymptotical formula of Laplace transform}

For $u_1,u_2\in\R$ and $T>0$, set local alternatives:
\[
  \vartheta_1 = \vartheta + \frac{u_1}{\sqrt{T}},
  \qquad
  \vartheta_2 = \vartheta + \frac{u_2}{\sqrt{T}},
  \qquad
  h:=\vartheta_2-\vartheta_1 = \frac{u_2-u_1}{\sqrt{T}}.
\]
Let $L_T^{\mathrm{mix}}(a,\vartheta_1,\vartheta_2)$ be the Laplace transform defined in \eqref{eq:Laplace-mixed}.

\begin{remark}
Throughout this section, we work with the asymptotic mixed drift matrix $\mathcal{A}_{\infty}^{\mathrm{mix}}(t)$ introduced in \eqref{eq:A-infty-def}. The approximation error is $o(\|A(t)\|)$ and, under the integrability assumptions of Lemma~\ref{lem:psi-asymptotics}, does not affect the exponential growth rate of the determinant or the LAN limit.
\end{remark}

From the calculations in Section~\ref{sec:laplace} and the limit identification in Section~\ref{sec:LAN}, we have the asymptotic formula:
\begin{equation}\label{eq:lim-Laplace}
L_T^{\mathrm{mix}}(a,\vartheta_1,\vartheta_2)
=
\exp\left\{
-\frac12\int_0^T \mathrm{tr}\,\mathcal{A}_{\infty}^{\mathrm{mix}}(t)\,\mathrm{d}\brN{t}
\right\}
\bigl(\det\Xi_{1,\infty}(T)\bigr)^{-1/2}\,(1+o(1)),
\qquad T\to\infty.
\end{equation}
where we replace the exact coefficients by their limits, resulting in a relative error $1+o(1)$ and here 
\begin{align}
  \frac{\mathrm{d}}{\mathrm{d}\brN{t}}\Xi_{1,\infty}(t)
  &=
  -\Xi_{1,\infty}(t) \mathcal{A}_{\infty}^{\mathrm{mix}}(t)
  + \frac{1}{4}a\mu^2\,\Xi_{2,\infty}(t)M(t),
  \qquad \Xi_{1,\infty}(0)=I_4,
  \label{eq:eq:Xi1-infty}\\[0.3em]
  \frac{\mathrm{d}}{\mathrm{d}\brN{t}}\Xi_{2,\infty}(t)
  &=
  \Xi_{1,\infty}(t)B(t)B(t)^{\top}
  + \Xi_{2,\infty}(t) \mathcal{A}_{\infty}^{\mathrm{mix}}(t)^{\top},
  \qquad \Xi_{2,\infty}(0)=0.
  \label{eq:eq:Xi2-infty}
\end{align}
\medskip

Now, we are prepared to deal with the Kronecker or Hamiltonian embedding. Recall from \eqref{eq:Abell-mixed} that
\[
  A(t)=b(t)\,\ell(t)^{\top},
  \qquad
  b(t)=\begin{pmatrix}1\\ \psi(t)\end{pmatrix},
  \qquad
  \ell(t)=\begin{pmatrix}\psi(t)\\ 1\end{pmatrix},
\]
and the fact
\[
  K(\vartheta_2)=2g(\vartheta_2)I_2,
  \qquad
  g(\vartheta):=\frac{1}{\alpha(\vartheta)+\vartheta},
  \qquad
  \alpha(\vartheta):=\sqrt{\vartheta^{2}+\mu^{2}},
\]
we have
\[
  K(\vartheta_2)\,A(t)=2g(\vartheta_2)\,b(t)\ell(t)^{\top}.
\]
Consequently, all $t$--dependent terms in the limiting mixed drift
matrix $\mathcal{A}_{\infty}^{\mathrm{mix}}(t)$ and in
$B(t)B(t)^{\top}$ inherit the same rank--one factor $A(t)$; collecting
the deterministic prefactors yields an $8\times 8$ linear Hamiltonian
system. More precisely, with $\mathbf{J}=\begin{bmatrix}
0 & 1 \\
1 & 0
\end{bmatrix}$ , the pair
$\bigl(\Xi_{1,\infty}(t),\,\Xi_{2,\infty}(t)\mathbf{J}\bigr)$ satisfies
\[
\frac{\mathrm{d}\left(\Xi_{1,\infty}(t),\, \Xi_{2,\infty}(t)\mathbf{J} \right)}{\mathrm{d}\brN{t}}
=
\left(\Xi_{1,\infty}(t),\, \Xi_{2,\infty}(t)\mathbf{J}\right)
\left(\Xi_{mix} \otimes \frac{1}{2} A(t) \right),
\]
where $\Xi_{mix}$ is the constant $4\times 4$ matrix
\begin{equation}\label{eq:Xi_matrix}
\Xi_{mix}=
\begin{pmatrix}
\vartheta_1 & 0 & 2\mu^2 g_1^2 & 2\mu^2 g_1 \Delta g \\
h & \alpha_2 & 2\mu^2 g_1 \Delta g & 2\mu^2 (\Delta g)^2 \\
0 & 0 & -\vartheta_1 & -h \\
0 & \frac{1}{2} a \mu^2 & 0 & -\alpha_2
\end{pmatrix}.
\end{equation}
Here we use the shorthand
\[
  \alpha_i:=\alpha(\vartheta_i)=\sqrt{\vartheta_i^2+\mu^2},
  \qquad
  g_i:=g(\vartheta_i)=\frac{1}{\alpha_i+\vartheta_i},
  \qquad
  \Delta g:=g_1-g_2,
\]
so that $\alpha_2=\alpha(\vartheta_2)$ and $g_1=g(\vartheta_1)$.

\subsection{Perturbation analysis of eigenvalues}

The asymptotic behaviour of the likelihood is governed by the eigenvalues of $\Xi_{mix}$ with positive real parts. Let these eigenvalues be $x_1(h)$ and $x_3(h)$, which are perturbations of $\vartheta_1$ and $\alpha_2$ respectively.

\begin{lemma}\label{lem:eig-exp}
Assume $|h|$ is sufficiently small and $\vartheta_1>0$. Then $x_1(h)$ and $x_3(h)$ admit the following second-order expansions:
\begin{align}
  x_1(h) &= \vartheta_1 + \frac{a\mu^2\,g_1^{\,2}}{2\vartheta_1}\,h^2 + o(h^2), \label{eq:x1-exp}\\
  x_3(h) &= \alpha_2 + \frac{a\mu^4}{2\alpha_2}\,\bigl(g'(\vartheta)\bigr)^2\,h^2 + o(h^2), \label{eq:x3-exp}
\end{align}
where $g'(\vartheta)=-\bigl(\alpha(\vartheta)(\alpha(\vartheta)+\vartheta)\bigr)^{-1}$.
\end{lemma}

\begin{proof}
We analyze the roots of the characteristic equation $\det(\Xi_{mix}-xI_4)=0$.

We first get the expansion for $x_3(h)$ via Schur Complement\cite{HJ1985}. Since $x_3$ is a perturbation of $\alpha_2$, and $\alpha_2$ is separated from $\vartheta_1$, we can isolate the $2\times 2$ block corresponding to indices $\{2,4\}$ (the observation subsystem). Using the Schur complement with respect to this block, the characteristic equation near $x \approx \alpha_2$ reduces to:
\[
\det\begin{pmatrix}
\alpha_{2}-x & 2\mu^{2}(\Delta g)^{2} \\
\frac{1}{2}a\mu^{2} & -\alpha_{2}-x
\end{pmatrix}
+ o(h^2) = 0.
\]
Substituting $\Delta g \approx -g'(\vartheta)h$ and $x=\alpha_2+\delta$, we find $\delta^2 + 2\alpha_2\delta - a\mu^4(g')^2 h^2 \approx 0$, which yields the expansion \eqref{eq:x3-exp}.

Then we analysis the expansion for $x_1(h)$ via Feedback Cycle Analysis. The eigenvalue $x_1$ corresponds to the drift parameter $\vartheta_1$. Let $x=\vartheta_1-\varepsilon$, where $\varepsilon$ is a small correction. Instead of a full block decomposition, we identify the dominant terms in the Leibniz expansion of the determinant.
\begin{itemize}
    \item \textbf{Diagonal contribution (Gap):} The first-order term in $\varepsilon$ arises from the product of the diagonal entries:
    \[
    \text{Diag} \approx \varepsilon(\alpha_2-\vartheta_1)(-2\vartheta_1)(-\alpha_2-\vartheta_1) = 2\vartheta_1\mu^2 \varepsilon.
    \]
    \item \textbf{Cycle contribution (Loop):} The leading perturbation term of order $h^2$ arises from the unique feedback cycle $1 \to 3 \to 4 \to 2 \to 1$ in the interaction graph of $\Xi_{mix}$. This cycle represents the feedback of the estimation error through the observation process:
    \[
    \text{Loop} = \underbrace{(2\mu^{2}g_{1}^{2})}_{1\to3} \cdot \underbrace{(-h)}_{3\to4} \cdot \underbrace{(\frac{1}{2}a\mu^{2})}_{4\to2} \cdot \underbrace{(h)}_{2\to1} = -a\mu^4 g_1^2 h^2.
    \]
\end{itemize}
Balancing the diagonal dominance with this feedback cycle ($\text{Diag} \approx \text{Loop}$) implies $2\vartheta_1\mu^2 \varepsilon \approx -a\mu^4 g_1^2 h^2$, which yields $\varepsilon = -\frac{a\mu^2 g_1^2}{2\vartheta_1} h^2$. Thus, we obtain \eqref{eq:x1-exp}.
\end{proof}

\subsection{Calculation of the Slope and identification of the Fisher information}

By Liouville's formula \cite{Arnold1992}, the exponential growth rate of the determinant is the sum of the positive eigenvalues of the Hamiltonian. Additionally, using the trace asymptotics derived from the rank-one factorization $A(t)=b(t)\ell(t)^{\top}$, the limiting log-Laplace exponent (the ``Slope'') is given by:
\begin{equation}\label{eq:slope-def}
\mathrm{Slope}(h)
:=
\lim_{T\to\infty}\frac{1}{T}\log L_T^{\mathrm{mix}}
=
\frac12\Bigl[(\vartheta_1+\alpha_2)-(x_1(h)+x_3(h))\Bigr].
\end{equation}
Substituting the expansions from Lemma~\ref{lem:eig-exp}, the zero-order terms cancel. Collecting the $h^2$ terms:
\[
\mathrm{Slope}(h) = -\frac{h^2}{2} \left( \frac{a\mu^2 g_1^2}{2\vartheta_1} + \frac{a\mu^4 (g'(\vartheta))^2}{2\alpha_2} \right) + o(h^2).
\]
Recalling that $h = (u_2-u_1)/\sqrt{T}$, we obtain the LAN limit:
\[
  \log L_T^{\mathrm{mix}}(a,\vartheta_1,\vartheta_2)
  \longrightarrow
  -\frac{a}{2}(u_2-u_1)^2\,I_{\mathrm{mix}}(\vartheta),
  \qquad T\to\infty.
\]
The Fisher information $I_{\mathrm{mix}}(\vartheta)$ decomposes into drift and observation components:
\begin{equation}\label{eq:Imix-decomp}
I_{\mathrm{mix}}(\vartheta)
=
\underbrace{\frac{\mu^2 g(\vartheta)^2}{2\vartheta}}_{I_{\mathrm{drift}}}
+
\underbrace{\frac{\mu^4(g'(\vartheta))^2}{2\alpha(\vartheta)}}_{I_{\mathrm{obs}}}.
\end{equation}
Using the algebraic identities $g(\vartheta)=1/(\alpha+\vartheta)$ and $g'(\vartheta)=-1/[\alpha(\alpha+\vartheta)]$, this simplifies to:
\begin{equation}\label{eq:Imix-closed}
I_{\mathrm{mix}}(\vartheta)
=
\frac{\alpha-\vartheta}{2\vartheta(\alpha+\vartheta)}
+
\frac{(\alpha-\vartheta)^2}{2\alpha^3}
=
\frac{1}{2\vartheta} - \frac{2\vartheta}{\alpha(\alpha+\vartheta)} + \frac{\vartheta^2}{2\alpha^3}.
\end{equation}
This expression coincides with the classical Kalman--Bucy Fisher information derived in \cite{BrousteKleptsyna2010}, confirming that the mixed fractional noise does not alter the asymptotic information content of the drift parameter.

\section{Simulation}\label{sec:sim}

We illustrate the asymptotic behavior of the MLE by Monte Carlo simulation.
For a fixed $(\vartheta_{\rm true},\mu,H)$ and horizon $T$, we simulate the model through its innovation-form state--observation representation. The kernel $g(s,t)$ (and the associated deterministic quantities
needed to evaluate the quadratic variation and coefficient functions) is
computed numerically using the procedure in \cite{CX2021}.

On a time grid $0=t_0<\cdots<t_n=T$, we generate independent Gaussian increments
of the driving martingales consistent with the prescribed quadratic variation,
and then propagate the discretized state and observation equations by an
Euler-type scheme. Given one simulated observation path, for each candidate
$\vartheta$ we run the corresponding discretized Kalman filter and evaluate the Gaussian prediction-error (pseudo) log-likelihood; the MLE $\widehat{\vartheta}_T$ is obtained by numerical maximization over $\vartheta>0$. 

\begin{figure}[h]
    \centering
    \includegraphics[width=0.8\textwidth]{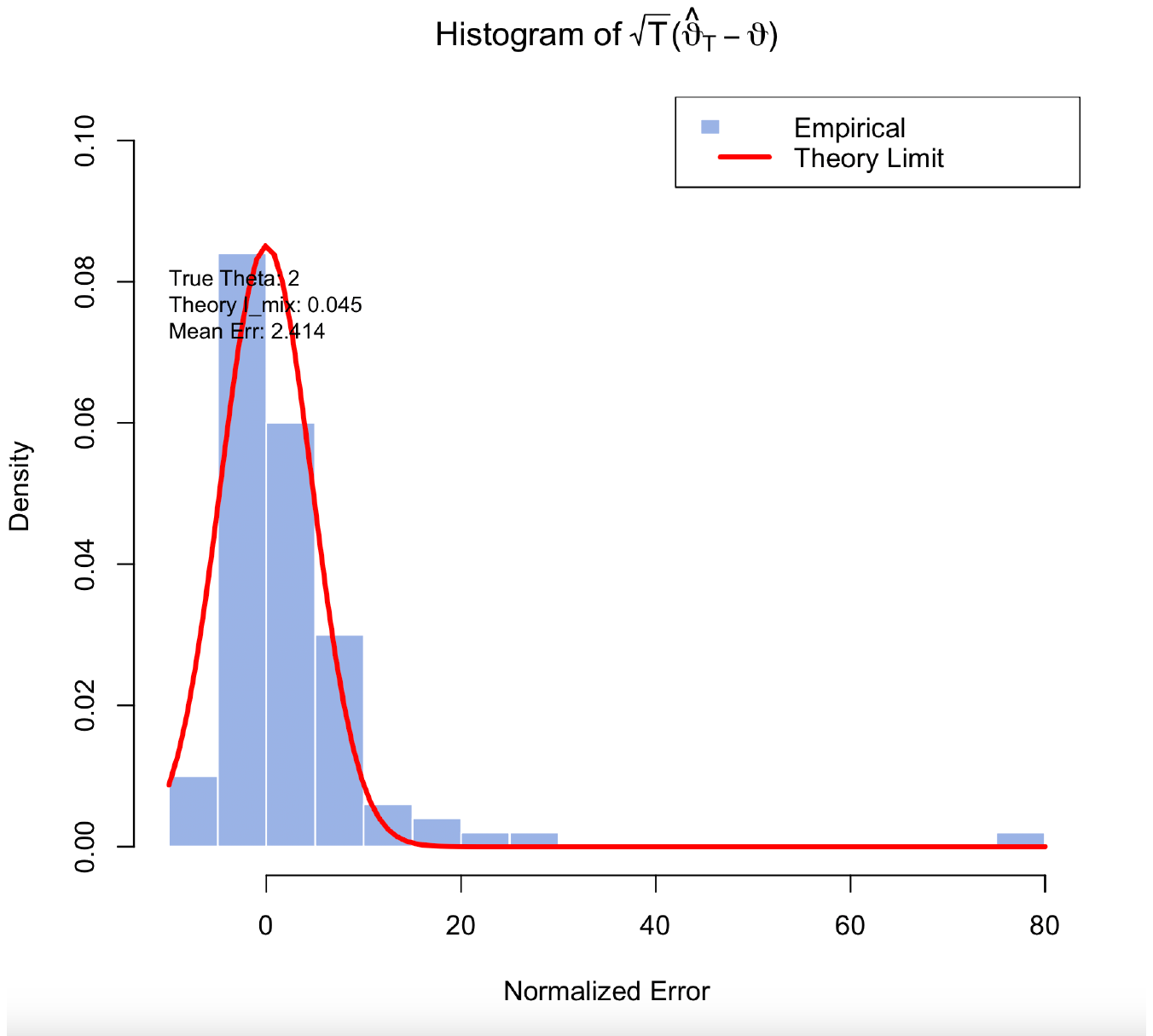} 
    \caption{Histogram of the standardized error $\sqrt{T}(\hat{\vartheta}_{T}-\vartheta)$ for $\vartheta=2.0, \mu=2.0, H=0.75$ with $T=100$. The red curve represents the theoretical asymptotic Gaussian density.}
    \label{fig:simulation}
\end{figure}

The simulation results are presented in Figure~\ref{fig:simulation}. The histogram of the standardized errors shows that the estimator is consistent. We compare the empirical statistics with the theoretical values derived from the Fisher information $\mathcal{I}(\vartheta)$.


\begin{thebibliography}{99}

\bibitem{BrousteKleptsyna2010}
A.~Brouste and M.~Kleptsyna.
\newblock Asymptotic properties of MLE for partially observed
fractional diffusion system.
\newblock {\em Statistical Inference for Stochastic Processes},
13(1):1--13, 2010.

\bibitem{Cheridito2001}
P.~Cheridito.
\newblock Mixed fractional Brownian motion.
\newblock {\em Bernoulli}, 7(6):913--934, 2001.

\bibitem{ChiganskyKleptsyna2017}
P.~Chigansky and M.~Kleptsyna.
\newblock Statistical Analysis of the Mixed Fractional Ornstein--Uhlenbeck Process.
\newblock {\em Theory of Probability and Its Applications},
63(3): 500--519, 2019.

\bibitem{KleptsynaLeBreton2002b}
M.~Kleptsyna and A.~Le~Breton.
\newblock Statistical analysis of the fractional Ornstein--Uhlenbeck
type process.
\newblock {\em Statistical Inference for Stochastic Processes},
5(3):229--248, 2002.

\bibitem{KleptsynaLeonenkoMuravlev2018}
P. Chigansky and M.~Kleptsyna.
\newblock Exact asymptotics in eigenproblems for fractional Brownian covariance operators.
\newblock {\em Stochastic Processes and their Applications},
128(6):2007-2059, 2018.

\bibitem{LiptserShiryayev1977}
R.~S.~Liptser and A.~N.~Shiryayev.
\newblock {\em Statistics of Random Processes I: General Theory}.
\newblock Springer, 1977.

\bibitem{Arnold1992} V. I. Arnold. 
\newblock {\em Ordinary Differential Equation}.
\newblock Springer Berlin, Heidelberg, 1992.

\bibitem{CX2021} C. Cai and W. Xiao.
\newblock {\em Simulation of an integro-differential equation and application in estimation of ruin probability with mixed fractional Brownian motion}.
\newblock Journal of Integral Equations and Applications,
33(1), 1-17, 2021.

\bibitem{HJ1985} R. A. Horn and C. R. Johnson.
\newblock {\em Matrix Analysis}.
\newblock Cambridge University Press; 1985.




\end{thebibliography}
\end{document}